\documentclass[12pt]{article}

\usepackage{amssymb,a4}
\usepackage{amsmath,amsfonts,amssymb,amsthm}

\newcommand{\1}{\mathbf {1}}
\newcommand{\Z}{{\mathbb Z}}
\newcommand{\C}{{\mathbb C}}

\newcommand{\wh}{{\widehat{\mathfrak h}}}

\newcommand{\J}{{\mathcal J}}

\newcommand{\al}{\alpha}

\newcommand{\dl}{\delta}
\newcommand{\gm}{\gamma}

\newcommand{\la}{\langle}
\newcommand{\ra}{\rangle}

\newtheorem{thm}{Theorem}[section]
\newtheorem{prop}[thm]{Proposition}
\newtheorem{lem}[thm]{Lemma}

\newtheorem{rmk}[thm]{Remark}
\newtheorem{definition}[thm]{Definition}

\begin{document}

\begin{center}
{\Large \bf   Quantum dimensions and fusion rules for parafermion vertex operator algebras}

\end{center}

\begin{center}
{ Chongying Dong$^{a}$\footnote{Supported by NSF grant DMS-1404741, NSA grant H98230-14-1-0118 and China NSF grant 11371261.}
and Qing Wang$^{b}$\footnote{Supported by
 China NSF grant (No.11371024), Natural Science Foundation of Fujian Province
(No.2013J01018) and Fundamental Research Funds for the Central
University (No.2013121001).}\\
$\mbox{}^{a}$ Department of Mathematics, University of
California, Santa Cruz, CA 95064\\
\vspace{.1cm}
$\mbox{}^{b}$School of Mathematical Sciences, Xiamen University,
Xiamen 361005, China\\
}
\end{center}

\begin{abstract}
The quantum dimensions  and the fusion rules for the parafermion vertex operator algebra associated to the
irreducible highest weight module for the affine Kac-Moody algebra
$A_1^{(1)}$ of level $k$ are determined.

\end{abstract}

\section{Introduction}
\def\theequation{1.\arabic{equation}}
\setcounter{equation}{0}
Parafermion vertex operator algebra is the commutant of the Heisenberg vertex operator algebra in the affine vertex operator algebra. It comes from a special kind of coset construction \cite{GKO}. Precisely speaking, let $L_{\widehat{\mathfrak{g}}}(k,0)$ be the level $k$
integrable highest weight module with weight zero for affine
Kac-Moody algebra $\widehat{\mathfrak{g}}$ associated to a finite
dimensional simple Lie algebra $\mathfrak{g}.$ Then $L_{\widehat{\mathfrak{g}}}(k,0)$
contains the Heisenberg vertex operator subalgebra
generated by the Cartan subalgebra $\mathfrak{h}$ of $\mathfrak{g}.$ The commutant
$K(\mathfrak{g},k)$ of the Heisenberg vertex operator subalgebra in $L_{\widehat{\mathfrak{g}}}(k,0)$ is the parafermion vertex operator
algebra. The structure and representation theory of parafermion vertex operator algebras has been widely studied these years(see \cite{DLY2}, \cite{DLWY}, \cite{DW1}, \cite{DW2}, \cite{ALY}). In particular, \cite{DW1} and \cite{DW2} show the role of parafermion vertex operator algebra $K(sl_2,k)$ in the parafermion vertex algebra $K(\mathfrak{g},k)$ is similar to the role of $3$-dimensional simple Lie algebra $sl_2$  played in Kac-Moody Lie algebras, we denote $K(sl_2,k)$ by $K_0$ in this paper. Moreover, it was proved that $K_0$ coincides with a certain $W$-algebra in \cite{DLY2} and \cite{DLWY}. Later in \cite{ALY}, the $C_2$-cofiniteness of parafermion vertex operator algebra $K(\mathfrak{g},k)$ has been established by proving the $C_2$-cofiniteness of parafermion vertex operator algebra $K_0$, and irreducible modules for parafermion vertex operator algebra $K_0$ were also classified therein. In the recent paper \cite{DR}, the rationality of parafermion vertex operator algebra $K(\mathfrak{g},k)$ was also obtained, moreover, the irreducible modules for $K(\mathfrak{g},k)$ were classified.

The notion of quantum dimensions of modules for vertex operator algebras was introduced in \cite{DJX}. It was proved therein for rational and $C_2$-cofinite vertex operator algebras, quantum dimensions do exist. In this paper, we first determine the quantum dimensions for the parafermion vertex operator algebra $K_0$. Then by using the important formula obtained in \cite{DJX} which shows that quantum dimensions are multiplicative under tensor product, we give the fusion rules for the parafermion vertex operator algebra $K_0$.

The paper is organized as follows. In Section 2, we recall some results about parafermion vertex operator algebra $K_0$. In Section 3, after reviewing the notion and properties of quantum dimensions of modules for vertex operator algebras, we give the quantum dimensions of parafermion vertex operator algebra $K_0$. In the final section, we obtained the fusion rules of parafermion vertex operator algebra $K_0$ by using the results of quantum dimensions of parafermion vertex operator algebra $K_0$.

\section{Preliminary}
\label{Sect:V(k,0)}

In this section, we recall from \cite{DLY2}, \cite{DLWY} and \cite{ALY} some basic results on the parafermion vertex
operator algebra associated to the
irreducible highest weight module for the affine Kac-Moody algebra
$A_1^{(1)}$ of level $k$ with $k\geq 2$ being an integer. First we recall the notion of the parafermion vertex operator algebra.

We are working in the setting of \cite{DLY2}. Let $\{ h, e, f\}$
be a standard Chevalley basis of $sl_2$ with brackets $[h,e] = 2e$, $[h,f] = -2f$,
$[e,f] = h$. Let $\widehat{sl}_2 = sl_2 \otimes \C[t,t^{-1}]
\oplus \C C$ be the affine Lie algebra associated to $sl_2$. Let $k \ge 2$
be an integer and
\begin{equation*}
V(k,0) = V_{\widehat{sl}_2}(k,0) = \mbox{Ind}_{sl_2 \otimes \C[t]\oplus \C
C}^{\widehat{sl}_2}\C
\end{equation*}
be an induced $\widehat{sl}_2$-module such that $sl_2 \otimes \C[t]$ acts
as $0$ and $C$ acts as $k$ on $\mathbf{1}=1$. We denote by $a(n)$ the operator on $V(k,0)$ corresponding to the action of
$a \otimes t^n$. Then
\begin{equation}\label{eq:affine-commutation}
[a(m), b(n)] = [a,b](m+n) + m \la a,b \ra \delta_{m+n,0}k
\end{equation}
for $a, b \in sl_2$ and $m,n\in \Z$. It is well known that there is a vertex
operator algebra structure on $V(k,0)$ and
it has a
unique maximal ideal $\J$, which is generated by a weight $k+1$ vector
$e(-1)^{k+1}\1$ \cite{Kac}. The quotient algebra $L(k,0) = V(k,0)/\J$ is the
simple vertex operator algebra associated to an affine Lie algebra
$\widehat{sl}_2$ of type $A_1^{(1)}$ with level $k$. The subspace $V_{\wh}(k,0)$ of $V(k,0)$ spanned by $h(-i_1) \cdots
h(-i_p)\1$ for $i_1 \ge \cdots \ge i_p \ge 1$ and $p \ge 0$ is a vertex
operator subalgebra of $V(k,0)$ associated to the Heisenberg algebra.
The parafermion vertex operator algebra $K_0$ is defined as the commutant of
$V_{\wh}(k,0)$ in $L(k,0)$, that is,
\begin{equation*}
K_0 = \{v \in L(k,0)\,|\, h(m)v = 0 \text{ for } m
\ge 0\}.
\end{equation*}
It was proved that $K_0$ is a simple vertex operator algebra
 and the irreducible $K_0$-modules $M^{i,j}$ for $0\leq i\leq k, 0\leq j\leq k-1$ were constructed in \cite{DLY2}. Note that $K_0=M^{0,0}$. It was also proved that $M^{i,j}\cong M^{k-i,k-i+j}$ as $K_0$-module in
\cite[Theorem 4.4]{DLY2} and moreover, Theorem 8.2 in \cite{ALY} showed that the $\frac{k(k+1)}{2}$ irreducible $K_0$-modules $M^{i,j}$ for $0\leq i\leq k, 0\leq j\leq i-1$ constructed in \cite{DLY2} form a complete set of isomorphism classes of irreducible $K_0$-module. Moreover, $K_0$ is $C_2$-cofinite \cite{ALY} and rational \cite{DR}.

Recall from \cite{DLY2} that $L = \Z\al_1 + \cdots + \Z\al_k$ with $\la \al_p,\al_q\ra = 2\dl_{pq}$. $V_L = M(1)
\otimes \C[L]$ is the lattice vertex operator algebra associated with the lattice $L$.  Let $\gamma = \al_1 + \cdots +
\al_k$. Thus $\la \gamma,\gamma \ra = 2k$. It is well known that the vertex operator algebra associated with a
positive definite even lattice is rational \cite{Dong}. And any irreducible modules for the lattice vertex operator algebra $V_{\Z\gm}$ is isomorphic
to one of $V_{\Z\gm + n\gm/2k}$, $0 \le n \le 2k-1$ \cite{Dong}. Let $L(k,i)$ for $0\leq i\leq k$ be the irreducible modules for the rational vertex operator algebra $L(k,0)$. The following result was due to \cite{DLY2}.

\begin{lem}\label{lem:dec}
$L(k,i) = \oplus_{j=0}^{k-1} V_{\Z\gm + (i-2j)\gm/2k} \otimes M^{i,j}$ as
$V_{\Z\gm}$-modules.
\end{lem}

\section{Quantum dimensions for irreducible $K_0$-modules
}\label{Sect:maximal-ideal-tI}
In this section, we first recall the notion and some basic facts about quantum dimension from \cite{DJX}. Then we determine the quantum dimensions of the irreducible $K_0$-modules.

Let $\left(V,Y,1,\omega\right)$ be a vertex operator algebra (see
\cite{FLM}, \cite{LL}). We define weak module, module and admissible module following \cite{DLM1, DLM2}. Let $W\left\{ z\right\} $
denote the space of $W$-valued formal series in arbitrary complex
powers of $z$ for a vector space $W$.

\begin{definition}A \emph{weak $V$-module} $M$ is
a vector space with a linear map
\[
Y_{M}:V\to\left(\text{End}M\right)\{z\}
\]

\[
v\mapsto Y_{M}\left(v,z\right)=\sum_{n\in\mathbb{Z}}v_{n}z^{-n-1}\ \left(v_{n}\in\mbox{End}M\right)
\]
which satisfies the following conditions for $u,v\in V, w\in M$:

\[
u_{n}w=0\ {\rm for} \ n\gg0,
\]

\[
Y_{M}\left(\mathbf{1},z\right)=Id_{M},
\]

\[
z_{0}^{-1}\text{\ensuremath{\delta}}\left(\frac{z_{1}-z_{2}}{z_{0}}\right)Y_{M}\left(u,z_{1}\right)
Y_{M}\left(v,z_{2}\right)-z_{0}^{-1}\delta\left(\frac{z_{2}-z_{1}}{-z_{0}}\right)Y_{M}\left(v,z_{2}\right)Y_{M}\left(u,z_{1}\right)
\]

\[
=\delta\left(\frac{z_{1}-z_{0}}{z_{2}}\right)Y_{M}\left(Y\left(u,z_{0}\right)v,z_{2}\right),
\]
 where $\delta\left(z\right)=\sum_{n\in\mathbb{Z}}z^{n}$. \end{definition}

\begin{definition}

A \emph{$V$-module} is a weak $V$-module\emph{
}$M$ which carries a $\mathbb{C}$-grading $M=\bigoplus_{\lambda\in\mathbb{C}}M_{\lambda},$
where $M_{\lambda}=\{w\in M|L(0)w=\lambda w\}$ and $L(0)$ is one of the coefficient operators of $Y(\omega,z)=\sum_{n\in\mathbb{Z}}L(n)z^{-n-2}.$
Moreover we require
that $\dim M_{\lambda}$ is finite and for fixed $\lambda,$ $M_{\lambda+n}=0$
for all small enough integers $n.$

\end{definition}

\begin{definition}An \emph{admissible $V$-module} $M=\oplus_{n\in\mathbb{Z}_{+}}M\left(n\right)$
is a $\mathbb{Z}_{+}$-graded weak module
such that $u_{m}M\left(n\right)\subset M\left(\mbox{wt}u-m-1+n\right)$
for homogeneous $u\in V$ and $m,n\in\mathbb{Z}.$ $ $

\end{definition}


\begin{definition}A vertex operator algebra $V$ is called \emph{rational}
if the admissible module category is semisimple. \end{definition}

The following lemma about rational vertex operator algebras is
well known \cite{DLM3}.

\begin{lem} If $V$ is rational and $M$ is an irreducible
admissible $V$-module, then

(1) $M$ is a $V$-module and there exists a $\lambda\in\mathbb{C}$
such that $M=\oplus_{n\in\mathbb{Z_{+}}}M_{\lambda+n}$
where $M_{\lambda}\neq0.$ And $\lambda$ is called the conformal weight
of $M;$

(2) There are only finitely many irreducible admissible
$V$-modules up to isomorphism. \end{lem}

\begin{definition} We say that a vertex operator algebra $V$ is
\emph{$C_{2}$-cofinite} if $V/C_{2}(V)$ is finite dimensional, where
$C_{2}(V)=\langle v_{-2}u|v,u\in V\rangle.$ \end{definition}


Let $M=\bigoplus_{\lambda\in \C}M_{\lambda}$ be a $V$-module. Set $M'=\bigoplus_{\lambda\in \C}M_{\lambda}^{*}$, the restricted dual of $M$£¬ where $M_{\lambda}^{*}=\mbox{Hom}_{\mathbb{C}}(M_{\lambda},\mathbb{C}).$
It was proved in \cite{FHL} that $M^{'}$ is naturally a $V$-module where the vertex
operator $Y_{M'}(v,z)$ is defined for $v\in V$ via
\begin{eqnarray*}
\langle Y_{M'}(v,z)f,u\rangle=\langle f,Y_{M}(e^{zL(1)}(-z^{-2})^{L(0)}v,z^{-1})u\rangle,
\end{eqnarray*}
where $\langle f,w\rangle=f(w)$ is the natural paring $M'\times M\to\mathbb{C}.$ The $V$-module $M^{'}$ is called the \emph{contragredient
module} of $M$. A $V$-module $M$ is called \emph{self-dual} if $M$ and $M^{'}$ are isomorphic $V$-modules. The following result was proved in \cite{L3}.


\begin{lem}\label{selfdual} Let $V$ be a simple vertex operator algebra such that $L(1)V_{(1)}\neq V_{(0)}$. Then $V$ is self-dual.
\end{lem}

\begin{rmk}\label{K0selfdual} Note that the weight one subspace of $K_0$ is zero. By using lemma \ref{selfdual}, parafermion vertex operator algebra $K_{0}$ is obviously self-dual.
\end{rmk}




Now we  recall from \cite{FHL} the notions of intertwining operators and fusion rules.

\begin{definition} Let $(V,\ Y)$ be a vertex operator algebra and
let $(W^{1},\ Y^{1}),\ (W^{2},\ Y^{2})$ and $(W^{3},\ Y^{3})$ be
$V$-modules. An \emph{intertwining operator} of type $\left(\begin{array}{c}
W^{3}\\
W^{1\ }W^{2}
\end{array}\right)$ is a linear map
\[
I(\cdot,\ z):\ W^{1}\to\text{\ensuremath{\mbox{Hom}(W^{2},\ W^{3})\{z\}}}
\]

\[
u\to I(u,\ z)=\sum_{n\in\mathbb{Q}}u_{n}z^{-n-1}
\]
 satisfying:

(1) for any $u\in W^{1}$ and $v\in W^{2}$, $u_{n}v=0$ for $n$
sufficiently large;

(2) $I(L(-1)v,\ z)=\frac{d}{dz}I(v,\ z)$;

(3) (Jacobi identity) for any $u\in V,\ v\in W^{1}$

\[
z_{0}^{-1}\delta\left(\frac{z_{1}-z_{2}}{z_{0}}\right)Y^{3}(u,\ z_{1})I(v,\ z_{2})-z_{0}^{-1}\delta\left(\frac{-z_{2}+z_{1}}{z_{0}}\right)I(v,\ z_{2})Y^{2}(u,\ z_{1})
\]
\[
=z_{2}^{-1}\left(\frac{z_{1}-z_{0}}{z_{2}}\right)I(Y^{1}(u,\ z_{0})v,\ z_{2}).
\]

The space of all intertwining operators of type $\left(\begin{array}{c}
W^{3}\\
W^{1}\ W^{2}
\end{array}\right)$ is denoted by
$$I_{V}\left(\begin{array}{c}
W^{3}\\
W^{1}\ W^{2}
\end{array}\right).$$ Let $N_{W^{1},\ W^{2}}^{W^{3}}=\dim I_{V}\left(\begin{array}{c}
W^{3}\\
W^{1}\ W^{2}
\end{array}\right)$. These integers $N_{W^{1},\ W^{2}}^{W^{3}}$ are usually called the
\emph{fusion rules}. \end{definition}




\begin{definition} Let $V$ be a vertex operator algebra, and $W^{1},$
$W^{2}$ be two $V$-modules. A module $(W,I)$, where $I\in I_{V}\left(\begin{array}{c}
\ \ W\ \\
W^{1}\ \ W^{2}
\end{array}\right),$ is called a \emph{tensor product} (or fusion product) of $W^{1}$
and $W^{2}$ if for any $V$-module $M$ and $\mathcal{Y}\in I_{V}\left(\begin{array}{c}
\ \ M\ \\
W^{1}\ \ W^{2}
\end{array}\right),$ there is a unique $V$-module homomorphism $f:W\rightarrow M,$ such
that $\mathcal{Y}=f\circ I.$ As usual, we denote $(W,I)$ by $W^{1}\boxtimes_{V}W^{2}.$
\end{definition}

\begin{rmk}
It is well known that if $V$ is rational, then for any two irreducible
$V$-modules $W^{1}$ and $W^{2},$ the fusion product $W^{1}\boxtimes_{V}W^{2}$ exists and
$$
W^{1}\boxtimes_{V}W^{2}=\sum_{W}N_{W^{1},\ W^{2}}^{W}W,
$$
 where $W$ runs over the set of equivalence classes of irreducible
$V$-modules.
\end{rmk}




Now we recall some notions about quantum dimensions.

\begin{definition}Let $M=\oplus_{n\in\mathbb{Z}_{+}}M_{\lambda+n}$
be a $V$-module, the \emph{formal character} of $M$
is defined as

\[
\mbox{ch}_{q}M=\mbox{tr}_{M}q^{L\left(0\right)-c/24}=q^{\lambda-c/24}\sum_{n\in\mathbb{Z}_{+}}\left(\dim M_{\lambda+n}\right)q^{n},
\]
where $c$ is the central charge of the vertex operator algebra $V$
and $\lambda$ is the conformal weight of $M$. \end{definition}

It is proved \cite{Z,DLM4} that $\mbox{ch}_{q}M$ converges to a
holomorphic function in the domain $|q|<1.$ We denote the holomorphic
function $\mbox{ch}_{q}M$ by $Z_{M}\left(\tau\right)$. Here and
below, $\tau$ is in the upper half plane $\mathbb{H}$ and $q=e^{2\pi i\tau}$.

Let $M^{0},\cdots,M^{d}$ be the inequivalent irreducible $V$-modules
with corresponding conformal weights $\lambda_{i}$ and $M^{0}\cong V$.



From the Remark 2.13 of \cite{DJX}, we have

\begin{definition}
\[
Z_{M^{i}}\left(-\frac{1}{\tau}\right)=\sum_{j=0}^{d}S_{i,j}Z_{M^{j}}\left(\tau\right).
\]
The matrix $S=\left(S_{i,j}\right)_{i,j=0}^{d}$ is called an \emph{$S$-matrix}.
\end{definition}


The following definition of quantum dimension was introduced in \cite{DJX}.

\begin{definition} \label{quantum dimension}Let $V$ be a vertex
operator algebra and $M$ a $V$-module such that $Z_{V}\left(\tau\right)$
and $Z_{M}\left(\tau\right)$ exist. The quantum dimension of $M$
over $V$ is defined as
\[
\mbox{qdim}_{V}M=\lim_{y\to0}\frac{Z_{M}\left(iy\right)}{Z_{V}\left(iy\right)},
\]
 where $y$ is real and positive. \end{definition}
The following result was obtained in \cite[Lemma 4.2]{DJX}.

 \begin{lem}\label{quan dim Si0/S00}Let $V$ be a simple, rational
and $C_{2}$-cofinite vertex operator algebra of $CFT$ type with
$V\cong V'$. Let $M^{i}$  for $0\leq i\leq d$ be the inequivalent irreducible $V$-modules
with corresponding conformal weights $\lambda_{i}$ and $M^{0}\cong V$.
Assume $\lambda_{0}=0$ and $\lambda_{i}>0\ \forall i\not=0$.
Then $\mbox{qdim}_{V}M^{i}=\frac{S_{i,0}}{S_{0,0}}$.
\end{lem}

From now on, we assume $V$ is a rational, $C_{2}$-cofinite vertex
operator algebra of CFT type with $V\cong V'$. Let $M^{0}\cong V,\, M^{1},\,\cdots,\, M^{d}$
denote all inequivalent irreducible $V$-modules. Moreover, we assume
the conformal weights $\lambda_{i}$ of $M^{i}$ are positive for
all $i>0.$ From Remark \ref{K0selfdual} and statement in Section 2, the parafermion vertex operator algebra $K_0$
satisfies all the assumptions.

The following result shows that the quantum dimensions are multiplicative under tensor product \cite{DJX} .

\begin{prop}\label{quantum-product} Let $V$ and $M_i$ for $0\leq i\leq d$ be as in Lemma \ref{quan dim Si0/S00}. Then
\[
\mbox{qdim}_{V}\left(M^{i}\boxtimes M^{j}\right)=\mbox{qdim}_{V}M^{i}\cdot \mbox{qdim}_{V}M^{j}
\]
for $i,\, j=0,\cdots,\, d.$
\end{prop}
 Before giving the main result of this section, we recall the following character of irreducible $K_0$-modules $M^{i,j}$ which is given in \cite{CGT,GQ}:
 \begin{equation*}
\mbox{ch} M^{i,j} = \eta(\tau) c^i_{i-2j}(\tau)
\end{equation*}
by \cite[(3.34)]{GQ}. Note that $k$, $l$ and $m$ in \cite{GQ} are $k$, $i$
and $i-2j$, respectively in our notation.

\begin{thm}\label{quantum-dimension} The quantum dimensions for all irreducible $K_0$-modules $M^{m,n}$ are
$$\mbox{qdim}_{K_{0}}M^{m,n}=\frac{\sin\frac{\pi(m+1)}{k+2}}{\sin\frac{\pi}{k+2}}$$
for $0\leq m\leq k,\, 0\leq n \leq k-1,$ where $M^{m,n}$ are the irreducible modules of $K_0$ constructed in \cite{DLY2}.
\end{thm}

\begin{proof} Let $M^{m,n}(\tau)$ denote the character of $M^{m,n}$ for $0\leq m\leq k,\ 0\leq n\leq k-1$. The $S$-modular transformation of characters has the following form \cite{GQ}, \cite{Kac}:
$$M^{m,n}(-\frac{1}{\tau})=\sum\limits_{m^{'},n^{'}}S_{m,n}^{m^{'},n^{'}}M^{m^{'},n^{'}}(\tau),$$
where $S_{m,n}^{m^{'},n^{'}}=(k(k+2))^{-\frac{1}{2}}\exp\frac{i\pi(m-2n)(m^{'}-2n^{'})}{k}\sin\frac{\pi(m+1)(m^{'}+1)}{k+2}$. From \cite{DLY2} and \cite{ALY}, we see that $K_0$ has $\frac{k(k+2)}{2}$ irreducible modules $M^{m,n}$  with the conformal weights $$\lambda_{m,n}=\frac{1}{2k(k+2)}(k(m-2n)-(m-2n)^{2}+2kn(m-n+1))$$
for $0\leq m\leq k,\ 0\leq n\leq m-1$. It is easy to check that $\lambda_{k,0}=0$ and $\lambda_{m,n}>0$ for $(m,n)\neq (k,0).$ Thus by using Lemma \ref{quan dim Si0/S00}, we have $$\mbox{qdim}_{K_{0}}M^{m,n}=\frac{S_{m,n}^{0,0}}{S_{0,0}^{0,0}}=\frac{\sin\frac{\pi(m+1)}{k+2}}{\sin\frac{\pi}{k+2}}.$$

\end{proof}

\section{Fusion rule for irreducible $K_0$-modules
}\label{Sect: fusion product}\def\theequation{4.\arabic{equation}}
\setcounter{equation}{0}

In this section, we give the fusion rules for irreducible $K_0$-modules.
 First we fix some notations. Let $W^{1},W^{2},W^{3}$ be irreducible $K_0$-modules. In the following, we use $I\left(\begin{array}{c}W^3\\
W^{1}\,W^{2}\end{array}\right)$ to denote the space  $I_{K_0}\left(\begin{array}{c}W^3\\
W^{1}\,W^{2}\end{array}\right)$ of all intertwining operators
of type $\left(\begin{array}{c}W^3\\
W^{1}\,W^{2}\end{array}\right)$, and use $W^{1}\boxtimes W^{2}$ to denote the fusion product $W^{1}\boxtimes_{K_0}W^{2}$
for simplicity.

We recall the fusion rules for affine vertex operator algebra of type $A_1^{(1)}$ \cite{TK} for later use.

\begin{lem}\label{affine} $$L(k,i)\boxtimes_{L(k,0)} L(k,j)=\sum\limits_{c} L(k,c),$$
where $|i-j|\leq c\leq i+j, \ i+j+c\in 2\mathbb{Z},\ i+j+c\leq 2k.$\\
\end{lem}

\begin{thm}\label{para-fusion} The fusion rule for the irreducible modules of Parafermion vertex operator algebra $K_0$ is as follows:
\begin{eqnarray}
M^{i,i^{'}}\boxtimes M^{j,j^{'}}=\sum\limits_{c}  M^{c,\overline{\frac{1}{2}(2i^{'}-i+2j^{'}-j+c)}},\label{eq:4.1}
\end{eqnarray}
where $\overline{a}$ means the residue of the integer $a$ modulo $k$, $0\leq i,j\leq k, 0 \leq i^{'},j^{'}\leq k-1,$ $|i-j|\leq c\leq i+j, \ i+j+c\in 2\mathbb{Z},\ i+j+c\leq 2k.$ Moreover, with fixed $i,i^{'},j,j^{'}$, $M^{c,\overline{\frac{1}{2}(2i^{'}-i+2j^{'}-j+c)}}$ for $|i-j|\leq c\leq i+j, \ i+j+c\in 2\mathbb{Z},\ i+j+c\leq 2k$ are inequivalent irreducible modules.
\end{thm}

\begin{proof} We take $V=L(k,0), U=V_{\mathbb{Z}\gamma}\otimes K_0$ in Proposition 2.9 of \cite{ADL}, from the Lemma \ref{lem:dec}, we see that
$$\mbox{dim}I_{V}\left(\begin{array}{c}L(k,l)\\
L(k,i)\,L(k,j)\end{array}\right)\leq \mbox{dim}I_{U}\left(\begin{array}{c}L(k,l)\\
V_{\mathbb{Z}\gamma+\frac{(i-2i^{'})\gamma}{2k}}\otimes M^{i,i^{'}}\,V_{\mathbb{Z}\gamma+\frac{(j-2j^{'})\gamma}{2k}}\otimes M^{j,j^{'}}\end{array}\right)$$
for $0\leq i,j,l\leq k, 0\leq i^{'},j^{'},l^{'}\leq k-1.$ Note that $L(k,l) = \oplus_{j=0}^{k-1} V_{\Z\gm + (l-2j)\gm/2k} \otimes M^{l,j}$, thus we have
$$\mbox{dim}I_{U}\left(\begin{array}{c}L(k,l)\\
V_{\mathbb{Z}\gamma+\frac{(i-2i^{'})\gamma}{2k}}\otimes M^{i,i^{'}}\,V_{\mathbb{Z}\gamma+\frac{(j-2j^{'})\gamma}{2k}}\otimes M^{j,j^{'}}\end{array}\right)$$
$$=\sum_{l^{'}=0}^{k-1} \mbox{dim}I_{U}\left(\begin{array}{c}V_{\mathbb{Z}\gamma+\frac{(l-2l^{'})\gamma}{2k}}\otimes M^{l,l^{'}}\\
V_{\mathbb{Z}\gamma+\frac{(i-2i^{'})\gamma}{2k}}\otimes M^{i,i^{'}}\,V_{\mathbb{Z}\gamma+\frac{(j-2j^{'})\gamma}{2k}}\otimes M^{j,j^{'}}\end{array}\right).$$

Then by using Theorem 2.10 of \cite{ADL}, we have
$$\mbox{dim}I_{U}\left(\begin{array}{c}V_{\mathbb{Z}\gamma+\frac{(l-2l^{'})\gamma}{2k}}\otimes M^{l,l^{'}}\\
V_{\mathbb{Z}\gamma+\frac{(i-2i^{'})\gamma}{2k}}\otimes M^{i,i^{'}}\,V_{\mathbb{Z}\gamma+\frac{(j-2j^{'})\gamma}{2k}}\otimes M^{j,j^{'}}\end{array}\right)$$
$$=\mbox{dim}I_{V_{\mathbb{Z}\gamma}}\left(\begin{array}{c}V_{\mathbb{Z}\gamma+\frac{(l-2l^{'})\gamma}{2k}}\\
V_{\mathbb{Z}\gamma+\frac{(i-2i^{'})\gamma}{2k}}\,V_{\mathbb{Z}\gamma+\frac{(j-2j^{'})\gamma}{2k}}\end{array}\right)\cdot \mbox{dim}I_{K_0}\left(\begin{array}{c} M^{l,l^{'}}\\
M^{i,i^{'}}\,M^{j,j^{'}}\end{array}\right).$$
Recall that the fusion rule for lattice vertex operator algebras is:
$$V_{\mathbb{Z}\gamma+\lambda}\boxtimes_{V_{\mathbb{Z}\gamma}} V_{\mathbb{Z}\gamma+\mu}=V_{\mathbb{Z}\gamma+\lambda+\mu}$$
for $\lambda, \mu\in(\mathbb{Z}\gamma)^{\circ}$, where $(\mathbb{Z}\gamma)^{\circ}$ is the dual lattice of $\mathbb{Z}\gamma$, this together with the fusion rule for affine vertex operator algebra given in Lemma \ref{affine}, we see that if $l^{'}\neq \frac{l-i+2i^{'}-j+2j^{'}}{2}$, then
$$\mbox{dim}I_{U}\left(\begin{array}{c}V_{\mathbb{Z}\gamma+\frac{(l-2l^{'})\gamma}{2k}}\otimes M^{l,l^{'}}\\
V_{\mathbb{Z}\gamma+\frac{(i-2i^{'})\gamma}{2k}}\otimes M^{i,i^{'}}\,V_{\mathbb{Z}\gamma+\frac{(j-2j^{'})\gamma}{2k}}\otimes M^{j,j^{'}}\end{array}\right)$$
$$=\mbox{dim}I_{V_{\mathbb{Z}\gamma}}\left(\begin{array}{c}V_{\mathbb{Z}\gamma+\frac{(l-2l^{'})\gamma}{2k}}\\
V_{\mathbb{Z}\gamma+\frac{(i-2i^{'})\gamma}{2k}}\,V_{\mathbb{Z}\gamma+\frac{(j-2j^{'})\gamma}{2k}}\end{array}\right)\cdot \mbox{dim}I_{K_0}\left(\begin{array}{c} M^{l,l^{'}}\\
M^{i,i^{'}}\,M^{j,j^{'}}\end{array}\right)=0.$$
And we also have $$\mbox{dim}I_{K_0}\left(\begin{array}{c} M^{c,\overline{\frac{1}{2}(2i^{'}-i+2j^{'}-j+c)}}\\
M^{i,i^{'}}\,M^{j,j^{'}}\end{array}\right)\geq 1.$$ In the following, we use the quantum dimension of the irreducible modules over $K_0$ to prove that $$\mbox{dim}I_{K_0}\left(\begin{array}{c} M^{c,\overline{\frac{1}{2}(2i^{'}-i+2j^{'}-j+c)}}\\
M^{i,i^{'}}\,M^{j,j^{'}}\end{array}\right)=1.$$
Since from Proposition \ref{quantum-product}, we know \[
\mbox{qdim}_{K_0}\left(M^{i,i^{'}}\boxtimes M^{j,j^{'}}\right)=\mbox{qdim}_{K_0}M^{i,i^{'}}\cdot \mbox{qdim}_{K_0}M^{j,j^{'}},
\]  and we also have

\[
M^{i,i^{'}}\boxtimes M^{j,j^{'}}=\sum_{(0,0)\leq(l,l^{'})\leq(k,l-1)}N_{(i,i^{'}),(j,j^{'})}^{(l,l^{'})}M^{l,l^{'}}.
\] Since Theorem \ref{quantum-dimension} shows that $$\mbox{qdim}_{K_{0}}M^{i,i^{'}}=\frac{\sin\frac{\pi(i+1)}{k+2}}{\sin\frac{\pi}{k+2}}.$$
Thus we only need to prove the following identity holds:
$$\frac{\sin\frac{\pi(i+1)}{k+2}\cdot\sin\frac{\pi(j+1)}{k+2}}{\sin\frac{\pi}{k+2}\cdot{\sin\frac{\pi}{k+2}}}
=\sum\limits_{c}\frac{\sin\frac{\pi(c+1)}{k+2}}{\sin\frac{\pi}{k+2}}$$
for $0\leq i,j\leq k, 0 \leq i^{'},j^{'}\leq k-1$, $|i-j|\leq c\leq i+j, \ i+j+c\in 2\mathbb{Z},\ i+j+c\leq 2k.$  This identity is equivalent to
the following identity:
\begin{eqnarray}\cos\frac{\pi(i+j+2)}{k+2}-\cos\frac{\pi(i-j)}{k+2}=\sum\limits_{c}(\cos\frac{\pi(c+2)}{k+2}-\cos\frac{\pi c}{k+2}).\label{eq:4.2}
\end{eqnarray}
We note that if $i+j<k$, then $c_{\mbox{min}}=i-j$, $c_{\mbox{max}}=i+j$, thus (\ref{eq:4.2}) holds.
If $i+j>k$, $c_{\mbox{min}}=i-j$, $c_{\mbox{max}}=i+j-2n$ for some $n$ satisfying that $i+j-2n+i+j=2k$, that is, $c_{\mbox{max}}=2k-i-j$, thus (\ref{eq:4.2}) also holds.

Now we prove that the modules $M^{c,\overline{\frac{1}{2}(2i^{'}-i+2j^{'}-j+c)}}$ appeared in the sum (\ref{eq:4.1}) are not isomorphic to each other. If  $M^{c,\overline{\frac{1}{2}(2i^{'}-i+2j^{'}-j+c)}}\cong M^{c^{'},\overline{\frac{1}{2}(2i^{'}-i+2j^{'}-j+c^{'})}}$ for some $c, c^{'}$ satisfying $|i-j|\leq c\leq i+j, \ i+j+c\in 2\mathbb{Z},\ i+j+c\leq 2k$. Since we know that $M^{i,j}\cong M^{k-i,k-i+j}$ as $K_0$-module for $0\leq i\leq k, 0\leq j\leq k-1$ and the $\frac{k(k+1)}{2}$ irreducible $K_0$-modules $M^{i,j}$ for $0\leq i\leq k, 0\leq j\leq i-1$ exhaust all the isomorphism classes of irreducible $K_0$-modules, it follows that $c^{'}=k-c$ and $$\overline{k-c+\frac{1}{2}(2i^{'}-i+2j^{'}-j+k-c)}=\overline{\frac{1}{2}(2i^{'}-i+2j^{'}-j+k-c)},$$ it is impossible by a direct calculation. This proves the assertion.

\end{proof}

\end{document}